\documentclass[11pt]{amsart}
\pdfoutput=1
\usepackage{graphicx,pinlabel}
\usepackage[all]{xy}
\usepackage[mathscr]{eucal}
\usepackage{microtype}
\usepackage{amssymb,amsmath}
\usepackage[top=1in, bottom=1in, left=1in, right=1in]{geometry}
\usepackage{lpic}
\usepackage{caption,subcaption}
\usepackage{array}
\usepackage{cancel, xcolor, paralist, soul, stmaryrd}
\usepackage{float}
\usepackage{hyperref}

\newcommand{\br}{\mathbb R}
\newcommand{\bh}{\mathbb H}
\newcommand{\bn}{\mathbb N}
\newcommand{\bz}{\mathbb Z}

\newcommand{\bq}{\mathbb Q}

\newcommand{\cf}{\mathcal F} 
\newcommand{\cl}{\mathcal L} 
\newcommand{\cg}{\mathcal G}

\newcommand{\Aut}{\mathrm{Aut}} 
\newcommand{\PSL}{\mathrm{PSL}}

\hypersetup{
    bookmarksnumbered,
    unicode=false,          
    pdftoolbar=true,        
    pdfmenubar=true,        
    pdffitwindow=false,     
    pdfstartview={FitH},    
    pdftitle={Combinatorics of k-Farey graphs},    
    pdfauthor={Jonah Gaster, Miguel Lopez, Emily Rexer, Zo\"e Riell, Xiao Yang},     
    pdfsubject={Geometric topology, Elementary number theory},   
    pdfkeywords={Farey graph, Hyperbolic geometry, Curves on surfaces}, 
    pdfnewwindow=true,      
    colorlinks=true,       
    linkcolor=blue,          
    citecolor=red,        
    filecolor=magenta,      
    urlcolor=cyan          
}

\newtheorem{thm}{Theorem}

\newtheorem{prob}{Problem}
\newtheorem{prop}{Proposition}
\newtheorem{lem}{Lemma}
\newtheorem{cor}{Corollary}

\theoremstyle{definition}
\newtheorem{definition}{Definition}

\theoremstyle{remark}
\newtheorem{rem}{Remark}
\newtheorem{ex}{Example}

\numberwithin{equation}{section}

\newcommand{\modu}[1]{\ (\mathrm{mod}\ #1)}

\title[combinatorics of $k$-Farey graphs]{Combinatorics of $k$-Farey graphs}
\date{October 21, 2018}

\author{Jonah Gaster, Miguel Lopez, Emily Rexer, Zo\"e Riell, and Yang Xiao}

\address{Department of Mathematics, McGill University \\
Burnside Hall, 805 Sherbrooke Street West, Montreal, Quebec H3A 0B9}
\email{jbgaster@gmail.com\\}

\address{Math Deptartment, Boston University \\ 
111 Cummington Mall, Boston, MA 02215 }
\email{mlopez3@bc.edu\\}

\address{Mathematics and Computer Science, Emory University \\
Mail Stop: 1131-002-1AC, Atlanta, GA 30322}
\email{emily.rexer@gmail.com\\}

\address{Department of Mathematics and Statistics, Smith College \\
Burton Hall 115, Northampton, MA 01063}
\email{zoe.riell@gmail.com\\}

\address{Department of Mathematics, Brown University \\
Box 1917, 151 Thayer Street, Providence, RI 02912}
\email{yang\_xiao@brown.edu\\}

\begin{document}

\begin{abstract}
With an eye towards studying curve systems on low-complexity surfaces, we introduce and analyze the $k$-Farey graphs $\cf_k$ and $\cf_{\leqslant k}$, two natural variants of the Farey graph $\cf$ in which we relax the edge condition to indicate intersection number $=k$ or $\le k$, respectively. 

The former, $\cf_k$, is disconnected when $k>1$. In fact, we find that the number of connected components is infinite if and only if $k$ is not a prime power. 
Moreover, we find that each component of $\cf_k$ is an infinite-valence tree whenever $k$ is even, and $\Aut(\cf_k)$ is uncountable for $k>1$.

As for $\cf_{\leqslant k}$, Agol obtained an upper bound of $1+\min\{p:p\text{ is a prime}>k\}$ for both chromatic and clique numbers, and observed that this is an equality when $k$ is either one or two less than a prime.
We add to this list the values of $k$ that are three less than a prime equivalent to $11\modu{12}$, and we show computer-assisted computations of many values of $k$ for which equality fails.
\end{abstract}

\maketitle


\section{Introduction}

The Farey graph $\cf$ is constructed as follows: the vertices
\[
\left\{ \frac p q \ : \; p,q\in\bz, \; \gcd ( p,q ) =1 \right\} \cup \left\{ \frac 1 0 \right\} 
\]
are equipped with a determinant pairing between $p/q$ and $a/b$ given by
\[
d_\cf\left( \frac p q, \frac a b\right) =
\left| \det \left(
\begin{array}{cc}
p & a \\ 
q & b
\end{array}\right)
\right|,
\] 
and we add edges between $p/q$ and $a/b$ whenever $d_\cf(p/q,a/b)=1$. 
The Farey graph admits a description as the `curve graph' of some low-complexity surfaces, in which the vertices correspond to isotopy classes of simple closed curves, and in which edges correspond to the smallest possible intersection number on the surface---in the relevant cases, these are one or two.

Variations of the curve graph in which the edge condition is altered to take into account other intersection numbers has recently attracted attention \cite{Schaller,ICERM2,Aougab,JMM}. 
Motivated to study the smallest complexity nontrivial such examples,
we introduce the \emph{$k$-Farey graphs} $\cf_k$ and $\cf_{\leqslant k}$: These graphs have the same vertex set as $\cf$, but we form edges of $\cf_k$ (resp.~$\cf_{\leqslant k}$) whenever
\[
d_\cf \left( \frac p q, \frac a b \right)=  k \ \ \ \text{ ( resp.~ } \le k \text{ )}~. 
\]  

This paper investigates combinatorial properties of $\cf_k$ and $\cf_{\leqslant k}$. 
For both we exploit a natural map from the vertices of $\cf$ to lines in the finite plane $\left( \bz/ r\bz \right)^2$ (see \autoref{lines}). This reduction was first used by Agol, who was motivated by exploring boundary slopes of exceptional Dehn fillings for a finite-volume cusped hyperbolic 3-manifold \cite[Lemma 8.2]{Agol}. In our language, Agol found an upper bound for the chromatic number of $\cf_{\leqslant k}$ by choosing $r$ to be the smallest prime larger than $k$. 
For $\cf_k$ we find $r=k$ to be useful: when $k$ is a prime power these lines index the components.

\begin{figure}
	\centering
	\includegraphics[height=8cm]{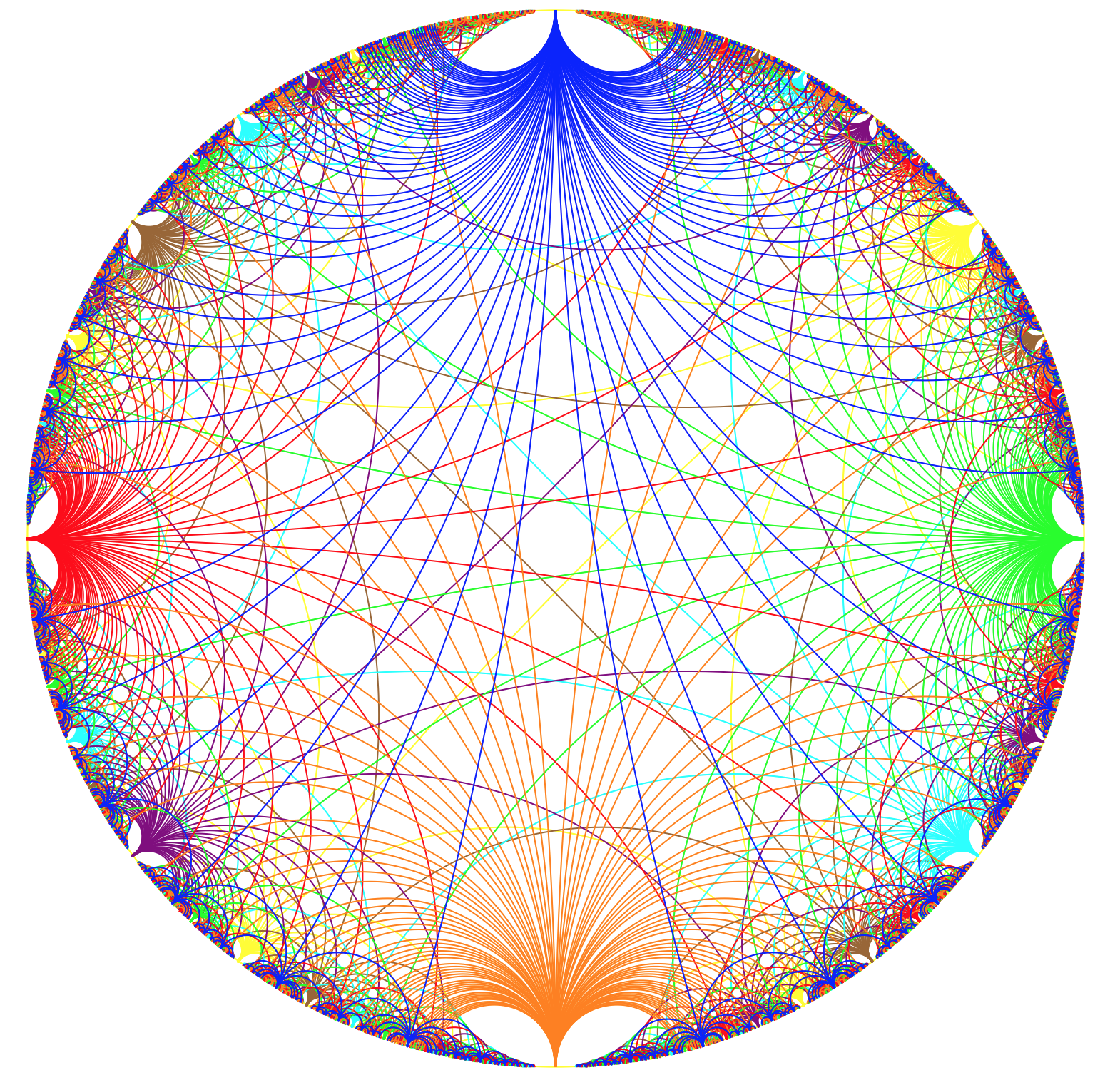}
	\caption{Pieces of $\cf_7$, colored according to the connected components.}
	\label{Farey7}
\end{figure}

\begin{thm}
\label{exact k}
The number of connected components of $\cf_k$ is given by \\
\[
b_0\left( \cf_k \right) =
\left \{ 
	\begin{array}{ll}
		p^{\ell-1}(p+1) & \text{ if }k=p^\ell, \text{ for }p \text{ a prime and }  \ell>0 \\
		\\
		\infty & \text{otherwise}.
	\end{array}
\right.
\]
\end{thm}
\bigskip

To prove connectivity of the preimages of lines when $k=p^\ell$, we identify basepoints and employ a reducing strategy to show that any point in the preimage of a line can be connected through a sequence of edges to the basepoint.

For other values of $k$, the proof of \autoref{exact k} is roughly as follows. We introduce an exhaustion 
\[ 
\cf_k^{(1)} \subset \cf_k^{(2)} \subset \ldots \subset \cf_k
\] satisfying:
\begin{enumerate}
\item For $m> k$, if two vertices are in distinct components of $\cf_k^{(m)}$, then they are in distinct components of $\cf_k^{(m+1)}$.
\item For $k$ not a prime power, $b_0\left(\cf_k^{(m)} \right) < b_0\left( \cf_k^{(m+1)}\right)$ for infinitely many $m$.
\end{enumerate}
Together, these evidently imply that $b_0(\cf_k)=\infty$ when $k$ is not a prime power.

The automorphism group of $\cf$ preserves $d_\cf$, and thus $\Aut(\cf) < \Aut(\cf_k)$.  
Because the action of $\PSL(2,\bz)$ on $\cf_k$ is vertex-transitive (see \autoref{auts sec}), it is easy to see that the connected components of $\cf_k$ are all isomorphic. 
In fact, the method of proof of \autoref{exact k} sheds more light on the structure of each component. 
For instance, each component is planar (see \autoref{planarity}), and when $k$ is even each component is an infinite-valence tree (see \autoref{tree}). 
When $k$ is odd, the latter statement is false; in that case $\cf_k$ contains triangles. 
Nonetheless, each vertex of $\cf_k$ is a cut vertex for $k>1$ (see \autoref{cut}).

It is not hard to see that $\Aut(\cf)$ contains $\PSL(2,\bz)$ as a finite-index subgroup. 
In stark contrast, a consequence of the above analysis is the following corollary:

\begin{cor}
\label{aut gp}
For $k>1$ the automorphism group $\Aut(\cf_k)$ is uncountable.
\end{cor}

The graph $\cf_k$ is the genus one version of an `exact $k$-curve graph' associated to a closed surface of genus $g$, where edges indicate geometric intersection number equal to $k$. 
In these terms, \autoref{aut gp} states that the automorphism group of the exact $k$-curve graph in genus one is quite different than the (extended) mapping class group of the torus for $k>1$, another low-complexity counterexample in Ivanov's `Meta-Conjecture' (see \cite{Ivanov}). 
However, the automorphism group of the exact $k$-curve graph in higher genus remains unclear, and the methods of this paper seem unlikely to shed light on this question; in particular it seems that the exact $k$-curve graph is connected for $g>1$.

In contrast, it has recently been shown that the automorphism group for the `at-most-$k$-curve graph' is isomorphic to the mapping class group when $g$ is sufficiently large relative to $k$ \cite{ICERM2}. 
The analogous statement for $\cf_{\leqslant k}$, namely that $\Aut(\cf_{\leqslant k}) \cong \Aut(\cf)$, seems within reach.

We turn our attention to $\cf_{\leqslant k}$. 
Let $p(k)$ equal the smallest prime larger than $k$. Agol showed the bounds $\omega(\cf_{\leqslant k}) \le \chi(\cf_{\leqslant k}) \le 1+ p(k)$ for all $k$ (see \cite[Lemma 8.2]{Agol}, cf.~\cite[Theorem 8.6]{GGV}), and he observed that these are equalities when $k$ is either one or two less than a prime.
As a consequence, he also observed that $\omega(\cf_{\leqslant k})$ is asymptotic to $k$ via the Prime Number Theorem \cite{Agol2}. 
We add to the families of $k$ for which Agol's bounds are equalities:

\begin{thm}
\label{up to k}
Suppose that $k$ is a number such that $p(k)=k+3$, and $p(k)\equiv 11 \modu{12}$.
Then we have equality of chromatic and clique numbers $\omega(\cf_{\leqslant k}) = \chi(\cf_{\leqslant k}) = 1+ p(k)$.
\end{thm}

The known cases in which $\omega(\cf_{\leqslant k}) = 1+p(k)$ have $\omega(\cf_{\leqslant k})=k+r$, for $r=2,3,4$, respectively, each for infinitely many values of $k$. 
We expect there are many other such constructions lurking in this area, and we propose the following problems about the sequence $\{\omega(\cf_{\leqslant k})\}$:

\begin{prob}[Exceptional $k$]
\label{exceptions}
Characterize $k$ for which $\omega(\cf_{\leqslant k}) < 1+p(k)$.
\end{prob}

\begin{prob}[Big Jumps]
\label{big jumps}
Let $r\ge 5$. Is it true that $ \{k: \omega(\cf_{\leqslant k}) \ge k+r\}  $ is nonempty? infinite?
\end{prob}

We have collected experimental data towards both problems (see \autoref{evidence}) by analyzing values of $k$ less than $100$. Intriguingly, for \autoref{exceptions} it seems that many primes that are two less than a composite are such exceptions (though $k=37$ is curiously absent from the exceptional list). 
For \autoref{big jumps}, the data verifies nonemptiness for $r\le 6$.

\subsection*{Acknowledgements}
The first author thanks Ser-Wei Fu and Peter Shalen for conversations that contributed to this project.
All the authors thank Ian Agol and Jeffrey Lagarias for helpful correspondence, and Tarik Aougab for his helpful ideas, enthusiasm, and support.
This material is based upon work supported by the National Science Foundation under Grant No.~DMS-1439786 while the authors were in residence at the Institute for Computational and Experimental Research in Mathematics in Providence, RI, during the \href{https://icerm.brown.edu/summerug/2018/}{\emph{Summer@ICERM 2018: Low Dimensional Topology and Geometry}} program.

\section{notation and background}
\subsection{Notation}
For a graph $G$, we indicate the number of connected components of $G$ by $b_0(G)$ (that is, the \emph{zero-th Betti number} of $G$). We indicate the vertices of $G$ by $V(G)$, and for $a,b\in V(G)$ we indicate the edge relation by $a\sim b$. Let $\omega(G)$ and $\chi(G)$ indicate the clique and chromatic numbers of $G$, respectively.

\subsection{Numerators, denominators and \texorpdfstring{$\mathbb{Z}^2$}{Z2}}
Vertices of $\cf$ are reduced fractions which evidently have numerators and denominators. 
It is often useful in our analysis to assume that denominators are positive; since $p/q=(-p)/(-q)$ nothing is lost in this assumption.
Vertices of $x/y\in\cf$ are also in correspondence with lines $\{(\lambda x,\lambda y):\lambda \in \bz\} \subset \bz^2$. 
Each such line has two primitive elements $(x,y)$ and $(-x,-y)$. We sometimes use $(x,y)$ and $x/y$ interchangeably for vertices of $\cf$, taking care that this identification is $2$-to-$1$ as remarked above.

\subsection{Automorphisms}
\label{auts sec}
The automorphism group $\Aut(\cf)$ admits a homomorphism to $\bz/2\bz$ according to whether it preserves or negates the determinant of a pair. 
Recall that the kernel is isomorphic to $\PSL(2,\bz)$, which acts on $V(\cf)$ via \emph{linear fractional transformations}: 
the action of an element of $\PSL(2,\bz)$ on $x/y \in V(\cf)$ is given by
\[
\left(
\begin{array}{cc}
a & b \\ 
c & d
\end{array}\right) \cdot \frac x y =
\frac {ax + by}{cx+dy}~.
\]
It is not hard to see that $\PSL(2,\bz)$ acts trasitively on $V(\cf)$.

\subsection{Elementary calculations}
We use the following elementary fact (see, e.g.~\cite[Lemma 2.4]{Hatcher} for the proof).

\begin{lem}
\label{elt num th}
Suppose that $a,b,n\in\bz$ with $\gcd(a,b)=1$, and that $(x_0,y_0)$ is a solution to $ax+by=n$. Then the set of integral solutions to this equation is given by $\{(x_0,y_0)+m(a,b):m\in\bz\}$.
\end{lem}

This has an easy consequence in our setting:

\begin{lem}
\label{neighbors in Fk}
Suppose that $a/b \sim x/y$ for vertices $a/b,x/y\in V(\cf_k)$. Then the neighbors of $a/b$ in $\cf_k$ all have the form $\displaystyle \frac{x +am}{y+bm}$ for $m\in \bz$.
\end{lem}

\begin{proof}
The points $(x,y)$ and $(-x,-y)$ are solutions to the equations $ay-bx=k$ and $ay-bx=-k$, and the lemma follows from \autoref{elt num th}.
\end{proof}

We will also make use of the following exercise:

\begin{lem}
\label{clique number k}
The clique number of $\cf_k$ is two if $k$ is even and three if $k$ is odd.
\end{lem}

\begin{proof}
Apply an element of $\PSL(2,\bz)$ that takes a vertex in a maximum clique to $1/0$. Now the neighbors of $1/0$ have the form $x/k$, and any two such vertices $x/k$ and $x'/k$ have 
\[
d_\cf(x/k,x'/k) = k|x-x'|~.
\] 
First, observe that at most two such vertices can be pairwise connected, since $x<x'<x''$ implies that $|x-x''|\ge 2$. 
Thus $\omega (\cf_k)\le 3$. 
If $k$ is even, then as the fractions are reduced we have $|x-x'|\ge 2$, and it is clear that $\omega(\cf_k)=2$. 
When $k$ is odd $\{1/0,1/k,2/k\}$ form a clique of size three.
\end{proof}

\bigskip

\section{From vertices of \texorpdfstring{$\cf$}{F} to the projective line over \texorpdfstring{$\bz/r\bz$}{Z/rZ}}
\label{lines}
We start by recalling the projective line over $\bz/r\bz$ (see e.g.~\cite{Herzer} for background).
An element $(a,b)\in\left(\bz/r\bz\right)^2$ is called \emph{admissible} if $(ra,rb)=0$ implies that $r=0$. 
Denote the admissible elements by $\left(\bz/r\bz\right)^2_{\mathrm{adm}}$, and
let 
\[
\cl_r := \left(\bz/r\bz\right)^2_{\mathrm{adm}} / \sim , \ \ \text{ where } v \sim \lambda v \text{ for } \lambda\in \left(\bz/r\bz\right)^*~.
\]

\begin{rem}
When $r=p$ is prime, $\bz/r\bz$ is a field, and we have the more concise description that $\cl_p$ is the projectivization of the vector space $\left(\bz/p\bz\right)^2$.
\end{rem}

The vertices of $\cf$ are in correspondence with relatively prime pairs $(p,q)\in\bz^2$, modulo the equivalence relation $(p,q)\equiv(-p,-q)$. 
Because $-1$ is a unit mod $r$ the following is well-defined:
\begin{definition}
\label{reducing map}
The reduction $\bz^2 \to \left( \bz/r\bz\right)^2$ induces the map
\[
\phi_r : V(\cf) \to \cl_r~. 
\]
\end{definition}

We will make use of the transitive action of $\PSL(2,\bz)$ on $V(\cf)$:

\begin{lem}
\label{transitive lines}
The action $\PSL(2,\bz) \curvearrowright V(\cf)$ descends to a transitive action on $\cl_r$.
\end{lem}

\begin{proof}
It is easy to see that the action in the quotient is well-defined.
Transitivity of the action on $V(\cf)$ immediately implies transitivity on the image $\phi_r(V(\cf))\subset \cl_r$. To finish, the latter is equal to $\cl_r$:
Choose $(x,y) \in \left(\bz/r\bz\right)^2_{\mathrm{adm}}$ and let $(x,y)=\lambda(x',y')$ where $\lambda=\gcd(x,y)$. Note that admissibility implies that $\gcd(\lambda,r)=1$, so $\lambda\in\left(\bz/r\bz\right)^*$, and $(x,y) \sim (x',y')$. Now $\gcd(x',y')=1$, so $\phi_r(x'/y')$ is the line through $(x,y)$ as claimed.
\end{proof}

\bigskip

\section{connectivity of \texorpdfstring{$\cf_k$}{Fk}}
Choosing $r=k$, we find that $\phi_k$ is constant on components of $\cf_k$:

\begin{lem}
\label{constant}
If $v,w\in V(\cf_k)$ are in the same component of $\cf_k$ then $\phi_k(v) = \phi_k(w)$.
\end{lem}

\begin{proof}
Suppose that $v\sim w$, and apply an element of $\PSL(2,\bz)$ that takes $v$ to the vertex $0/1$ and $w$ to $a/b$. Now since $v\sim w$ we have $|a|=k$. The reduction of this equation mod $k$ implies that $\phi_r(w) = [(0,1)]= \phi_r(v)$. The result follows easily.
\end{proof}

It is remarkable that the preimages of $\phi_k$ are connected when $k=p^\ell$. 

\begin{prop}
\label{connected}
When $k=p^\ell$, the subgraph induced on $\phi_k^{-1}(L)$ is connected for all $L\in\cl_k$.
\end{prop}

\begin{proof}
By \autoref{transitive lines} we may assume that $L$ is the line through $(1,1)$. Choose as a basepoint $v=1/1$, and let $a/b$ be another point in $\phi_k^{-1}(L)$.
Observe first that $b\ne 0$: because $\phi_r(a,b)=L=\phi_r(1,1)$, we find that $b$ is a unit mod $r$, so by \autoref{constant} $a/b$ and $1/0$ are in distinct components of $\cf_k$.
Thus we may assume that $b>0$.
We compute a path from $a/b$ to $v$ in two steps: 
We treat first the case where $b=1$, and when $b>1$ we show a path exists from $a/b$ to $c/1$ for some integer $c$.

If $b=1$, then $a$ must be congruent to $1 \modu {p^\ell}$ since $a/b\in \phi_k^{-1}(L)$, so let $a=mp^\ell+1$ for some integer $m$. 
One can easily check that the vertices $(mp^\ell+1)/1$ and $((m-1)p^\ell+1)/1$ are incident, which gives us a path from $a/b$ to $1/1$ by adding or subtracting multiples of $p^\ell$.

Now suppose that $b>1$. 
We will show that there is another vertex of $\phi^{-1}(L)$, incident to $a/b$, and with strictly smaller denominator. 
Repeating this reduction as necessary, we terminate with a path from $a/b$ to a vertex with denominator $1$ as claimed.

Let $x_0 / y_0$ be a neighbor of $ a / b$ in $\cf_k$. 
By \autoref{neighbors in Fk} 
the neighbors of $a/b$ are given by $(x_{0}+am)/(y_{0}+bm)$ for $m \in \mathbb{Z}$.
Observe that if $b$ divides $y_0$ then the equation $d_\cf(a/b,x_0/y_0)=k$ would imply $b$ is a zero divisor in $\mathbb{Z}/k\mathbb{Z}$.
However, as observed above the coordinate $b$ is necessarily a unit.
Therefore $b$ does not divide $y_0$, and the division algorithm provides two values $y=y_1$ and $y=y_2$ such that $bn + y = y_{0}$ and $0 < |y| < b$.  
Let the corresponding values of $n$ be $n_1$ and $n_2$.

We now have two candidates for vertices of $\phi_k^{-1}(L)$ incident to $a/b$ with smaller denominator; it remains to check that at least one of them is a vertex in $\cf_k$. 
In particular, we need to check that $\gcd(x_0-an_i,y_0-bn_i)=1$ for some $i$. 
If $\gcd(x_0-an_i,y_0-bn_i) \ne 1$ then the common divisor is divisible by $p$, so it suffices to show that $x_0-an_i$ and $y_0-bn_i$ are not both divisible by $p$. Note that the two values $n_1$ and $n_2$ are subsequent integers. Therefore, if $(x_0-an_i,y_0-bn_i)\equiv (0,0) \modu p$ for $i=1,2$, then $(a,b)\equiv(0,0) \modu p$, a contradiction. 
Thus one of these values of $n_i$ produces a vertex incident to $a/b$ with strictly smaller denominator.
\end{proof}

\subsection{Calculating \texorpdfstring{$b_0(\cf_k)$}{b0(Fk)} when \texorpdfstring{$k=p^\ell$}{k=p to the l}}
\label{connected prime power}

\begin{lem}
\label{counting lines}
When $k=p^\ell$, we have $|\cl_k| = p^{\ell-1}(p+1)$.
\end{lem}

\begin{proof}
Observe that there are $p^{2 \ell} - p^{2 (\ell -1)}$ admissible pairs
in $\left(\bz/k\bz\right)^2$. 
Let $\Psi:\left(\bz/k\bz\right)^2_\mathrm{adm} \to \cl_k$ be the projectivization (i.e.~quotient) map.
Each line $L \in \cl_k$ has the same number of preimages in $\left(\bz/k\bz\right)^2$ under $\Psi$. 
In particular, by definition  
\[
\Psi^{-1}(L) = \left\{ \lambda (a, b) \mid \lambda \in \left(\bz/ k \bz\right)^* \right\}, 
\]
for some $(a, b) \in \left(\bz/k\bz\right)^2_{\mathrm{adm}}$. Since $|\left(\bz/ k \bz\right)^*| = p^\ell - p^{\ell -1}$, it follows that 
$ |\Psi^{-1}(L)|=p^\ell - p^{\ell -1} $ for any $L \in \cl_k$, and therefore
\[ 
|\cl_k| = \frac{p^{2 \ell} - p^{2 (\ell -1)}}{p^\ell - p^{\ell -1}} = p^{\ell-1}(p+1). \qedhere
\] 
\end{proof}

Together, \autoref{constant}, \autoref{counting lines}, and \autoref{connected} immediately imply the first half of \autoref{exact k}.

\subsection{\texorpdfstring{$b_0(\cf_k) = \infty$}{b0(Fk)=infinite} when \texorpdfstring{$k$}{k} is not a prime power}

We start by introducing a partial order to the vertices of $\cf_k$:

\begin{definition}
\label{ordering}
We write $\displaystyle \frac a b \prec \frac p q$ when $\displaystyle \frac a b \sim \frac p q$ in $\cf_k$ and $|a| \le |p|$ and $|b| \le |q|$.
\end{definition}

It is easy to see that $\prec$ is reflexive, antisymmetric, and transitive, and so defines a partial order on $V(\cf_k)$. When $\frac a b \prec \frac p q$ we refer to $a/b$ as a \emph{predecessor} of $p/q$ and $p/q$ as a \emph{successor} of $a/b$.

\begin{lem}[Eventually Comparable]
\label{comparable}
Suppose that $a/b $ is a vertex of $ \cf_k$ such that $|a|+|b|>k$. Then any neighbor of $a/b$ in $\cf_k$ is comparable to $a/b$.
\end{lem}

\begin{proof}
We employ the following simple observation: if $m$ and $n$ are of the same sign then $|m-n|=||m|-|n||$ and $|m+n|=|m|+|n|$. 

Let $x/y$ be a neighbor of $a/b$ in $\cf_k$, i.e.~$|ay-bx|=k$, with $b,y>0$. 
If $a$ and $x$ have opposite signs then we have 
\[
|ay-bx| = |a|y + |x|b \ge |a| + b >k~,
\]
by the assumption on $a$ and $b$.
This contradiction implies that $a$ and $x$ have the same sign.

Let $s=|x|-|a|$ and $t=b-y$.
If $x/y$ and $a/b$ are not comparable, then $s$ and $t$ must be non-zero and of the same sign.
Thus we have
\[
|ay-bx| = | |a|y - b|x| | = | |a| (b-t) - b(s+|a|) | = | t|a| + sb | = |t| |a| + |s| b \ge |a| + b >k~.
\]
This contradiction implies that $x/y$ and $a/b$ are comparable.
\end{proof}

In fact, it is possible to classify the possibilities for predecessors of a vertex.
\clearpage

\begin{lem}
\label{predecessors}
The number of predecessors of a vertex are as follows:
\begin{enumerate}
\item If $k$ is even, each vertex has at most one predecessor.
\item If $k$ is odd, each vertex has at most two predecessors. Moreover, if $p/q$ does have two predecessors $a_i/b_i$, then $a_1/b_1 \sim a_2/b_2$ in $\cf_k$.
\end{enumerate}
\end{lem}

\begin{proof}
If $q=0$ then there is nothing to show; indeed in this case $p/q$ can have no predecessors. We thus assume $q>0$, and proceed by considering two possible cases: 
\begin{enumerate}[(i)]
\item\label{1/0} the vertex $1/0$ is a predecessor, or 
\item\label{not 1/0}the vertex $1/0$ is not a predecessor.
\end{enumerate}

In case \eqref{1/0} we have $q=k$, and the neighbors of $p/q$ have the form $(1+pm)/(0+mk)$ with $m\in\bz$ by \autoref{neighbors in Fk}. The inequality $|1+pm|\le |p|$ makes it clear that there are at most two such values of $m$, one of which is $m=0$. Moreover, if there are two then we have 
\[
d_{\mathcal{F}}\left( \frac{1+pm}{mk}, \frac{1}{0} \right) =  | mk|= k~,
\]
so that both of these predecessors are connected by an edge in $\cf_k$.

For case \eqref{not 1/0}, let $a_1/b_1$ be a predecessor of $p/q$.
By \autoref{neighbors in Fk} any other incident vertex to $p/q$ has the form $( a_1+pm)/(  b_1+qm)$ for some integer $m$. For such a vertex to be a predecessor we must have $| a_1+pm| \leq |p|$ and $|   b_{1}+qm| \leq q$. 
In this case by the triangle inequality we have
\[
q \ge | b_1 + qm| \ge q|m| - b_1 \ge q(|m|-1)~,
\]
so that $|m|\le 2$. Moreover, because $b_1$ and $q$ are positive we have $m\le0$.
This means that the three possible predecessors of $p/q$ are $\frac{2p-a_1}{2q-b_1}$, $\frac{p-a_1}{q-b_1}$, and $\frac{a_1}{b_1}$.
If the first of these is indeed a predecessor, we would have $q\ge b_1\ge q$ and $|p| \ge |a_1| \ge |p|$, so that $a_1/b_1 = -p/q$. In this case, neither of the other two possibilities are predecessors.
This narrows the possible predecessors of $p/q$ to $\frac{p-a_1}{q-b_1}$ and $\frac{a_1}{b_1}$. Moreover, it is immediate that 
\[
d_{\mathcal{F}}\left( \frac{p-a_1}{q-b_1},\frac{a_1}{b_1} \right) =  k~.
\]

Thus we conclude that in either case \eqref{1/0} or \eqref{not 1/0} the vertex $p/q$ has at most two predecessors $a_i/b_i$, and if it does have two then $a_1/b_1 \sim a_2/b_2$ in $\cf_k$.
If $k$ is even then \autoref{clique number k} implies that $p/q$ can have at most one of these as predecessors.
\end{proof}

We now introduce the exhaustion
\[
\cf_k^{(1)} \subset \cf_k^{(2)} \subset \ldots \subset \cf_k
\]
referred to in the introduction.

\begin{definition}
\label{filtration}
The \emph{level} of a vertex $p/q$ is $\max\{|p|,|q|\} $.
For each $m\in \bn$, let $\cf_k^{(m)}$ indicate the subgraph of $\cf_k$ induced on vertices whose level is at most $m$.
\end{definition}

As a consequence of \autoref{comparable} and \autoref{predecessors}, we obtain:

\begin{prop}[Eventual Monotonicity]
\label{monotone}
For $m> k$, if $v$ and $w$ are vertices in distinct components of $\cf_k^{(m)}$, then they are in distinct components of $\cf_k^{(m+1)}$.
\end{prop}

\begin{proof}
We make two observations: 
\begin{enumerate}
\item Any pair of vertices $a/b$ and $x/y$ at level $m+1$ are not incident in $\cf_k$, and
\item any vertices $u_1$ and $u_2$ that are not incident in $\cf_k^{(m)}$ cannot both be neighbors of a vertex at level $m+1$.
\end{enumerate}
Indeed, for the first, incident vertices at level $m+1$ are comparable by \autoref{comparable}, so suppose $a/b \prec x/y$; in either case $b=y=m+1$ or $|a|=|x|=m+1$ we have $|ay-bx|>k$.
For the second, by \autoref{comparable} such vertices would be comparable to the vertex at level $m+1$ and by \autoref{predecessors} they would be incident.

Now suppose towards contradiction that $v=v_0 \sim v_1 \sim \ldots \sim v_n =w$ is a shortest path in $\cf_k^{(m+1)}$ between $v$ and $w$. Evidently, there must exist a vertex $v_i$ at level $m+1$ in this path, since $v$ and $w$ are in distinct components of $\cf_k^{(m)}$. By the first observation above, both $v_{i-1}$ and $v_{i+1}$ must be in $\cf_k^{(m)}$, and by the second $v_{i-1} \sim v_{i+1}$. Thus there is a shorter path from $v$ to $w$. This contradiction implies that there is no path from $v$ to $w$ in $\cf_k^{(m+1)}$.
\end{proof}

It follows that $m\mapsto b_0\left( \cf_k^{(m)}\right)$ is non-decreasing for $m>k$. Finally, we are able to show that this sequence increases infinitely often. Precisely,

\begin{prop}
\label{isolated vertices}
If $k$ is not a prime power then there exist infinitely many levels $m$ containing vertices isolated in $\cf_k^{(m)}$.
\end{prop}

\begin{proof}
Let $p$ be a prime divisor of $k$ such that $k=p^{\ell}c$ for some $\ell>0$ and $c$ relatively prime to $p$. Choose $q \equiv c \modu p$ where $0<q<p^{\ell}$. 
We claim that the vertices $(q+p^{l}n)/p^{\ell}$ are isolated in $\cf_{k}^{(q+p^{\ell}n)}$ when $n$ is such that $q+p^{l}n>k$. 
For such $n$ by \autoref{comparable} we need only check that $(q+p^{l}n)/p^{\ell}$ has no predecessors or successors in $\cf_{k}^{(q+p^{\ell}n)}$.

Suppose $a/b \in \mathcal{F}_{k}^{(q+p^{\ell}n)}$ is a predecessor to $(q+p^{\ell}n)/p^{\ell}$. Then we compute 
\[ 
d_\cf\left( \frac{q+p^{\ell}n}{p^{\ell}}, \frac{a}{b} \right)  = qb + bp^{\ell}n -ap^{\ell} = k = p^{\ell}c
\]
Reducing the above equation mod $p^\ell$ immediately 
yields $b \equiv 0 \modu {p^\ell}$ since $q$ was chosen to be coprime to $p$. However, $a/b$ is a predecessor, so $b=p^{\ell}$. We cancel $p^{\ell}$ above and rearrange to get
\[ 
p^{\ell}n - a = c - q \equiv 0 \modu p
\]
Thus, $a \equiv 0 \modu {p}$ and $a$ and $b$ have a common factor. 
This contradiction implies that $(q+p^{\ell}n)/p^{\ell}$ has no predecessor.

Now suppose that $a/b$ is a successor to $(q+p^{\ell}n)/p^{\ell} \in \mathcal{F}_{k}^{(q+p^{\ell}n)}$. Then we have $a= q+p^{\ell}n$ and $b > p^{\ell}$. Let $b = p^{\ell} +t $ for some positive integer $t$. Again we compute  
$$d_\cf\left( \frac{q+p^{\ell}n}{p^{\ell}}, \frac{a}{b} \right) =
  d_\cf\left( \frac{q+p^{\ell}n}{p^{\ell}}, \frac{q+p^{\ell}n}{p^{\ell}+t} \right) =  
  tq+tp^{\ell}n = t(q+p^{\ell}) > k
$$
so these vertices cannot be incident. Thus, $(q+p^{\ell}n)/p^{\ell}$ has no successors or predecessors and is therefore isolated. 
\end{proof}

\begin{rem}
\autoref{isolated vertices} relies on the assumption that $k$ is not a prime power, and the conclusion would not hold without it: for $k=p^\ell$,
when $m$ is sufficiently large there are no isolated vertices at level $m$.
\end{rem}

\subsection{Subgraphs of \texorpdfstring{$\cf_k$}{Fk}}
We give more attention to the subgraphs of $\cf_k$ given by preimages under $\phi_k$. 
Choose a line $L\in \cl_k$.

\begin{prop}
\label{planarity}
The restriction of the natural embedding $\cf\to \overline{\bh}$ to the preimage $\phi_k^{-1}(L)$ is a planar embedding.
\end{prop}

\begin{rem}
See \autoref{cliques} for detail about the `natural embedding' of $\cf$; in particular the above proposition implies that each component of $\cf_k$ is planar.
\end{rem}

\begin{proof}
Suppose on the contrary that, for vertices $v,w,v',w' \in \phi_k^{-1}(L)$ with $v\sim w$ and $v'\sim w'$ in $\cf_{k}$, we have $v<v'<w<w'$. After applying an element of $\PSL(2,\bz)$ (and switching labels if necessary) we may assume that $v=0/1$ and $w=k/q$, and let $v'=x/y$ and $w'=a/b$.

Note that after this normalization, the line $L$ in the image goes through $(0,1)$, and hence for any element $c/d\in\phi_k^{-1}(L)$ we must have that $k$ divides $c$. Thus the inequality $x/y < k/q < a/b$ can be simplified to
\[
\frac {x'} y  < \frac 1 q < \frac {a'} b~,
\]
where $a=a'k$ and $x=x'k$ for $a',x'\neq 0$.
This implies $a'q>b$ and $y>x'q$. Let $a'q=b+t$ and $y=x'q+s$ for positive integers $s,t$. Because $a/b\sim x/y$ in $\cf_k$ (and $x/y<a/b$ by our normalization) we have 
\[
k=ay-bx=(a'y-bx')k=(a'(x'q+s)-(a'q-t)x')k=(a's+x't)k>k~. \qedhere
\]

\end{proof}

When $k$ is even, we characterize the isomorphism type of the graph $\cf_k$:

\begin{prop}
\label{tree}
For $k$ even, each component of $\cf_k$ is an infinite-valence tree.
\end{prop}

\begin{proof}
Suppose first that a component of $\cf_k$ contains an embedded cycle $C$ that includes some vertex whose level is larger than $k$. 
Among the vertices of $C$ at maximal level, let $v$ be the vertex that is maximal with respect to $\prec$. 
Observe that the two neighbors of $v$ in $C$ are both comparable to it by \autoref{comparable}, and so by maximality they are both predecessors, contradicting \autoref{predecessors}.
Thus any embedded cycle must have its vertex levels bounded by $k$.

Now suppose that $\cf_k$ contains an embedded cycle $C$ whose vertex levels are at most $k$. 
Choose a vertex $a/b\in C$ so that $b\ne 0$---if $C$ is an embedded cycle such a vertex exists. The element  
\[
A=
\left(
\begin{array}{cc}
1 & N \\ 
0 & 1
\end{array}\right)
\]
of $\PSL(2,\bz)$ acts by automorphisms of $\cf_k$, taking $(a,b)$ to $(a+bN,b)$. For $N$ large, $A\cdot C$ is an embedded cycle of $\cf_k$ with a vertex whose level is larger than $k$, reducing to the previous case.
\end{proof}

The previous lemma fails when $k$ is odd (e.g.~by \autoref{clique number k}). Nonetheless, we have:

\begin{prop}
\label{cut}
If $k>1$ every vertex of $\cf_k$ is a cut vertex.
\end{prop}

\begin{proof}
Again by transitivity of $\Aut(\cf)\curvearrowright V(\cf)$ we need only show this result for one vertex in $\cf_k$, so choose $v =0/1$ and let $L=\phi_k(v)$. 
We will show that there are no edges in $\phi_k^{-1}(L)$ between a positive vertex and a negative vertex.
As $1/0\notin \phi_k^{-1}(L)$, it follows that such vertices must be in distinct components of $\phi_k^{-1}(L)\setminus\{v\}$.

Indeed, suppose that $a/b,x/y\in\phi_k^{-1}(L)$ with $a,b,y>0$ and $x<0$. 
Immediately, we have $x \equiv a \equiv 0 \modu{k}$,
so that there are $a',x'>0$ with $a=a'k$ and $x=-x'k$. It follows that $ay-bx=k(a'y+bx') > k$.
\end{proof}

This is enough to conclude:

\begin{prop}
\label{autos}
For $k>1$, $\Aut(\phi_k^{-1}(L))$ is uncountable.
\end{prop}

\begin{proof}
We claim that for each vertex $v$ in a component $G\subset \cf_k$, there are infinitely many components of $G\setminus \{v\}$ that are isomorphic. The claim follows as the set of permutations of $\bn$ is uncountable.

Again, by transitivity of $\Aut(\cf)\curvearrowright V(\cf)$ we may assume that $v=0/1$. Note that the cyclic subgroup of $\PSL(2,\bz)$ generated by 
$A=\left(
\begin{array}{cc}
1 & 0 \\ 
1 & 1
\end{array}
\right)$ 
fixes $v$.
Moreover, for any $w\ne v$, the sequence $\{A^nw \}$ approaches $v$ from the positive side and $\{A^{-n}v\}$ approaches $v$ from the negative side.

Now choose a neighbor $w_0 \sim v$. Choosing appropriate powers of $A$, we find that there are $n_1,n_2\in\bn$ so that $A^{n_1}w_0<0$ and $A^{n_2}w_0>0$. Let $w=A^{n_1}w_0$ and $B=A^{n_2-n_1}$ (so that $A^{n_2}w_0 = Bw$). 
By the proof of \autoref{cut}, we find that $w$ and $Bw$ are in distinct components of $G\setminus \{v\}$, and similarly $B^nw$ and $B^{n+1}w$ are in distinct components of $G\setminus \{v\}$ for each $n\in \bz$. It follows that $\{B^nw\}$ picks out infinitely many components of $G\setminus \{v\}$: applying appropriate powers of $B$ to a pair $B^{m_1}w$ and $B^{m_2}w$, we may arrange that one becomes positive and the other negative. Moreover, all such components are isomorphic, as they differ by powers of $B$.
\end{proof}

\begin{rem}
It is interesting to note that there is a rich family of graph automorphisms of $\Aut(\cf_k)$ that preserves connected components of $\cf_k$ when $k=p^\ell$: The \emph{level $k$ congruence group} of $\PSL(2,\bz)$ is the kernel of $\PSL(2,\bz) \to \PSL(2,\bz/k\bz)$, which induces a trivial action on the lines $\cl_k$.
\end{rem}

\bigskip 

\section{Cliques in \texorpdfstring{$\cf_{\leqslant k}$}{F<=k} via subgraphs of the dual graph to \texorpdfstring{$\cf$}{F}}
Agol observed that a collection of simple closed curves on the torus that pairwise intersect $\le k$ times would have pairwise distinct images under $\phi_p$ \cite[Lemma 8.2]{Agol}, where $p=p(k)$ is the smallest prime larger than $k$, from which a bound on the chromatic number is immediate. We include Agol's simple argument for completeness:

\begin{lem}
\label{coloring}
We have $\chi(\cf_{\leqslant k})\le 1+p(k)$.
\end{lem}

\begin{proof}
Suppose that $r>k$. We claim that $\phi_r:V(\cf) \to \cl_r$ gives a proper coloring of the vertices of $\cf_{\leqslant k}$. Taking $r=p(k)$ completes the proof by \autoref{counting lines}.

Suppose that $\phi_r(a/b)=\phi_r(p/q)$ for vertices $a/b,p/q\in V(\cf)$. This implies that $(p,q)\equiv \lambda(a,b)$ for $\lambda \in \left(\bz/r\bz\right)^*$, and $aq\equiv bp \modu r$ follows readily. In particular, we have $d_\cf(a/b,p/q) =nr$ for a non-negative integer $n$. But $d_\cf(v,w) $ is non-zero when $v\ne w$. Thus $d_\cf(a/b,p/q) \ge r > k$ and we conclude $a/b \not \sim p/q$.
\end{proof}

\begin{rem}
\label{improve Agol}
The proof above actually implies that $\chi(\cf_{\leqslant k})$ is at most $\min \{ | \cl_r | : r > k\}$. It is unclear where there exists $k$ for which this is smaller than $1+p(k)$. 
\end{rem}

Since $\omega(G) \le \chi(G)$ for all graphs $G$, it remains to find lower bounds for $\omega(\cf_{\leqslant k})$, i.e.~constructions of large cliques in $\cf_{\leqslant k}$.
Evidently, any collection $S$ of vertices of $\cf$ will determine a clique in $\cf_{\leqslant k}$ when $k$ is large enough—in particular, when $k$ is at least the maximum determinant pairing among elements of $S$. 
The challenge in finding good lower bounds for $\omega(\cf_{\leqslant k})$ is to make sure this `large' is small relative to the number of vertices of $S$. 

We now describe how to organize our search for families of vertices according to their relative positions in $\cf$. 
We start by recalling the relationship between the Farey graph $\cf$, the hyperbolic plane $\bh$, and continued fraction expansions.

\begin{figure}[h]
	\centering
	\vspace{.5cm}
	\begin{lpic}{dual-graph(,8cm)}
	\LARGE
		\lbl[]{10,15;$H_{0/1}$}
	\end{lpic}
	\vspace{.5cm}
	\caption{The dual graph $\cg$ to $\cf$, and a horocycle $H_{0/1}$.}
	\label{dual pic}
\end{figure}
\ \vspace{.5cm}\\

\subsection{The natural planar embedding of the Farey graph}
\label{cliques}
The index-two subgroup $\PSL(2,\bz)$ of $\Aut(\cf)$ is isomorphic to a subgroup of the isometries of the hyperbolic plane $\bh$. This is no accident: 
The Farey graph may be embedded in a compactification of the hyperbolic plane $\overline{\bh} = \bh \cup \partial_{\infty}\bh$, in which $V(\cf) $ is in natural correspondence with $\bq \cup \{\infty\} \subset \br \cup \{\infty\} \cong \partial_\infty \bh$ and edges are realized as hyperbolic geodesics between points of $\partial _\infty \bh$. 
We refer to this as the \emph{natural embedding} of $\cf$ in the plane. 
One manifestation of the usefulness of this point of view is the following:

The embedding $\cf \subset \overline{\bh}$ determines a tiling of $\bh$ by ideal hyperbolic triangles, and the dual graph $\cg$ to this tiling is a trivalent tree in $\bh$. 
The complementary regions of $\cg$ are approximate horoballs—the boundary of the complementary regions is piecewise geodesic rather than horocyclic—and each region accumulates onto a single point in $V(\cf)\subset \partial_\infty \bh$, the `center' of the horoball. 
We indicate the union of the edges of $\cg$ incident to the (approximate) horoball centered at $p/q$ by $H_{p/q}\subset \cg$.

Note that when $H_{p/q}$ and $H_{a/b}$ intersect we must have $p/q \sim a/b$, i.e.~these vertices have determinant pairing $1$. More generally, disjoint sets $H_{p/q}$ and $H_{a/b}$ are connected by a unique geodesic path in the tree $\cg$, and the planar embedding $\cg\subset \bh$ allows us to interpret this geodesic as a sequence of paths contained in horocycles, interspersed by left- and right-turns in $\cg$. We refer to the sequence of positive integers so obtained as a `left-right sequence' associated to the geodesic in $\cg$. In fact we have:

\begin{thm}
\label{intersection numbers}
Suppose that $(l_1,\ldots,l_m)$ is a left-right sequence for a geodesic from $H_{p/q}$ to $H_{a/b}$, for vertices $p/q,a/b\in V(\cf)$. 
Then the determinant $ d_\cf\left(  p /q,  a/ b\right)$ is given by the numerator of the continued fraction expansion
\[
l_1+\cfrac{1}{ l_2 + \cfrac{1}{\ldots +\cfrac{1}{l_m +1}}}
\]
\end{thm}

The action $\PSL(2,\bz) \curvearrowright V(\cf)$ is transitive, $d_\cf$-preserving, and induces an automorphism that preserves left-right sequences of geodesics in $\cg$. Thus for the proof of this theorem it suffices to check the claim when $a/b=1/0$. In this case, $d_\cf(1/0,p/q)=q$, and the claim follows from \cite[Thm.~2.1, p.~31]{Hatcher}.
We encourage readers to consult \cite{Hatcher} for a beautiful detailed exposition.

\begin{figure}
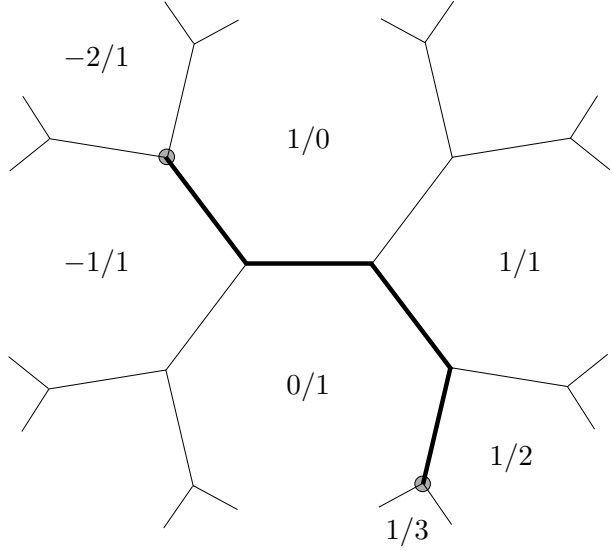

	\centering
	\vspace{.2cm}
	\begin{lpic}{example-intersection(8cm)}
		\lbl[]{170,80;$0/1$}
		\lbl[]{170,220;$1/0$}
		\lbl[]{290,150;$1/1$}
		\lbl[]{50,150;$-1/1$}
		\lbl[]{50,266;$-2/1$}
		\lbl[]{285,40;$1/2$}
		\lbl[]{226,-2;$1/3$}
	\end{lpic}
	\vspace{.2cm}
	\label{intersection pic}
	\caption{The computation of $d_\cf(1/3,-2/1)=7$.}
	\label{Farey pics} 
\end{figure}

\begin{ex}
\label{intersection ex}
The geodesic from $H_{-2/1}$ to $H_{1/3}$ is described by both left-right sequences $\{2,2\}$ and $\{1,1,2\}$ (see \autoref{Farey pics}); as well the geodesic from $H_{1/3}$ to $H_{-2/1}$ is given by $\{3,1\}$ and $\{1,2,1\}$. Note that we have 
\begin{align*}
\boxed{2} + \frac 1 {\boxed{2} +1}  &=  \frac {\boxed{7}} 3 \ \ , \hspace{1.1cm} \boxed{1} + \frac 1 {\boxed{1}+ \frac 1 {\boxed{2}+1} } = \frac {\boxed{7}} 4~, \\
\boxed{3} + \frac 1 {\boxed{1} +1}  &=  \frac {\boxed{7}} 2 \ \ \text{, and } \ \ \boxed{1} + \frac 1 {\boxed{2}+ \frac 1 {\boxed{1}+1} } = \frac {\boxed{7}} 5~,
\end{align*}
and $d_\cf(-2/1,1/3)=7$.
\end{ex}

The following observation is immediate from the calculation of \autoref{intersection numbers}.

\begin{lem}
\label{non-decreasing int}
If $\{a_1,\ldots,a_m\}$ and $\{b_1,\ldots,b_m\}$ are two left-right sequences with $a_i\le b_i$ for each $i$, then the intersection number determined by $\{a_i\}$ is at most that determined by $\{b_i\}$.
\end{lem}

Given a finite subgraph $K\subset \cg$, consider the associated subset $V_K\subset V(\cf)$ determined by 
\[
V_K = \left\{ \; \frac p q  \; : \; H_{p/q} \cap K \text{ is neither empty nor a singleton} \; \right \},
\]
and we let 
\[
I(K) = \max \left\{ d_\cf\left( \frac p q, \frac a b\right ) : \frac p q, \frac a b \in V_K\right\} .
\]
See \autoref{subgraph example} for an example. \bigskip\\

\begin{figure}
	\centering
	\includegraphics[width=6cm]{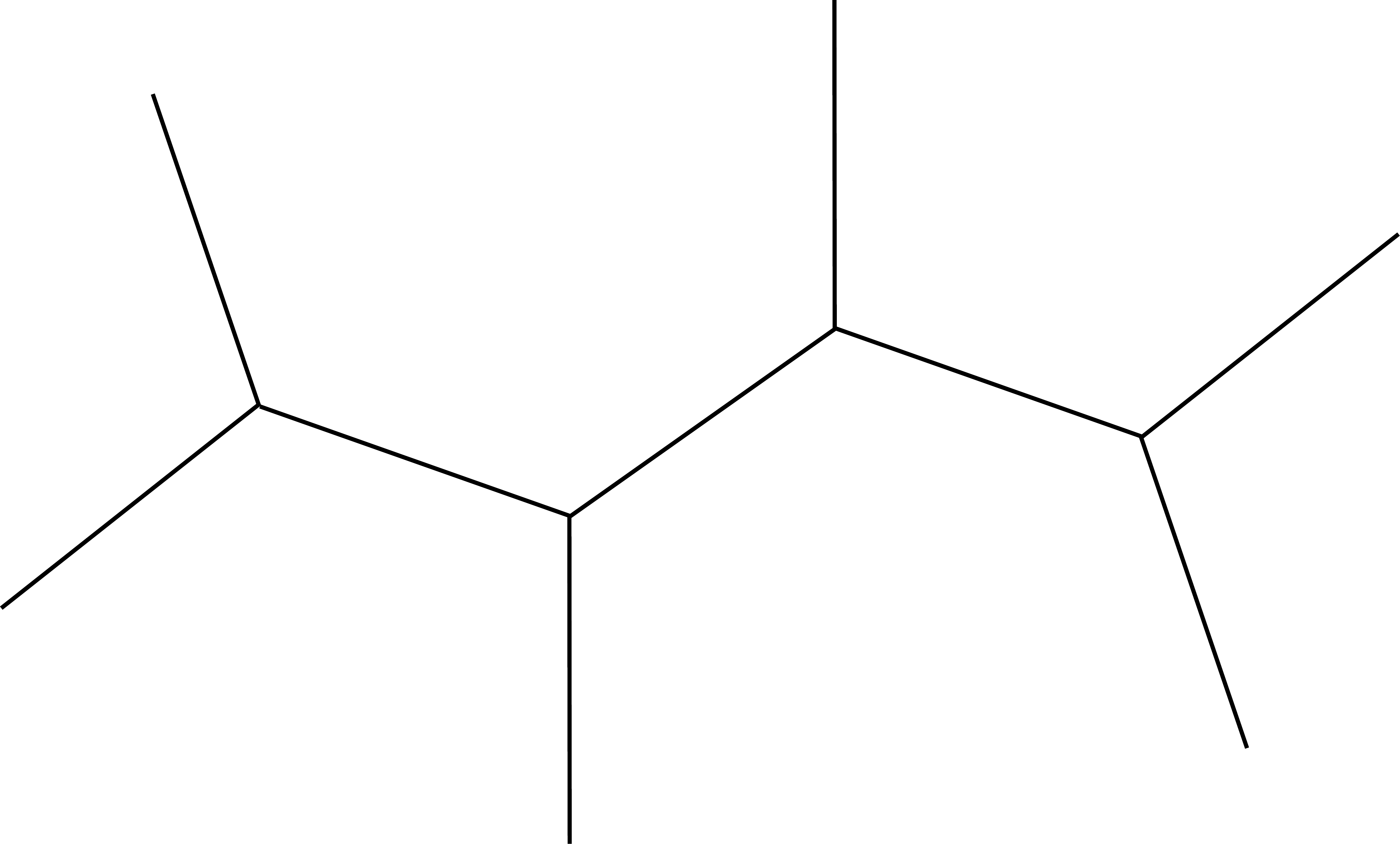}
	\caption{A subgraph $K$ of $\cg$ with $I(K) = 5$ and $|V_K|= 6$.}
	\label{subgraph example}
\end{figure}

We now describe several explicit constructions of finite subgraphs of $\cg$. See \autoref{subgraph pics} for visual presentations.

\begin{definition}
\label{families}
Let $P_n$ indicate an embedded path of length $n$ inside a horocycle of $\cg$.
\begin{enumerate}
\item Let $R_n$ be the union of $P_n$ with the $n-1$ edges of $\cg$ incident to non-endpoint vertices of $P_n$.
\item There is a vertex of $R_{2n}$ that is determined by the property that it is the unique fixed endpoint of an involution of $R_{2n}$. Let $S_n$ be the union of $R_{2n}$ with the two edges of $\cg$ incident to this endpoint.
\end{enumerate}
\end{definition}

\begin{lem}
\label{clique construction1}
We have $I(R_n) = n-1$, $|V_{R_n}| = n+1$, $I(S_n) = 2n-1$, and $|V_{S_n}| = 2n+2$.
\end{lem}

\begin{proof}
The left-right sequences between non-neighboring horoballs incident to $R_n$ are given by $\{\{i\}:1\le i\le n-2\}$. 
The extreme scenario is evidently $\{n-2\}$, so $I(R_n) = n-1$ by \autoref{non-decreasing int}.

For $S_n$, a set of left-right sequences for pairs of non-neighboring horoballs is given by $\{\{i\}:1\le i\le 2n-2\} \cup \{ \{j,1\}: 1\le j\le n-1\}$. 
It is not hard to see that both extreme cases give $I(S_n)=2n-1$.
\end{proof}

The third construction is the most involved. Start with $P_{7n-1}$, and number the non-endpoint vertices of $P_{7n-1}$ in order $1,\ldots,7n-2$. Now decompose these vertices into the sets
\begin{align*}
J_1 & = \{1,2,\ldots,2n-1\} \cup \{ 5n, 5n+1, \ldots, 7n-2\} , \\
J_2 & = \{2n, 2n+1, \ldots, 3n-1\} \cup \{ 4n, 4n+1, \ldots, 5n-1\} , \text{ and} \\
J_3 & = \{ 3n, 3n+1, \ldots, 4n-1\}.
\end{align*}
For each vertex $v\in J_1$, add the unique edge of $\cg \setminus P_{7n-1}$ incident to $v$. 
For each vertex $v\in J_2$, add the three edges of $\cg \setminus P_{7n-1}$ that are within distance one of $v$. 
Finally, for each $v\in J_3$, add the seven edges of $\cg \setminus P_{7n-1}$ that are within distance two of $v$. 
We refer to the result as $T_n$.

\begin{lem}
\label{clique construction2}
We have $I(T_n) = 12n-4$, and $|V_{T_n}| = 12n$.
\end{lem}

\begin{proof}
By \autoref{non-decreasing int}, it is enough to check the extreme left-right sequences for pairs of non-neighboring horoballs incident to $T_n$.
It is not hard to see that such extremes are given by the left-right sequences
\begin{align*}
& \{ 7n-3 \} , \{ 5n-2,1\}, \{4n-2,2\}, \{4n-2,1,1\},  \\
& \{2,3n-2,1\}, \{2,2n-2,2\}, \{2,2n-2,1,1\},  \\
& \{2,1,n-2,2\}, \{2,1,n-2,1,1\}, \{3,n-2,2\}, \{3,n-2,1,1\}~.
\end{align*}
These sequences determine continued fraction expansions with numerators $7n-2$, $10n-3$, $12n-5$, $12n-4$, $12n-4$, $12n-7$, $12n-5$, $9n-9$, $9n-6$, $9n-12$, and $9n-9$, respectively, so that $I(T_n)=12n-4$, as claimed. 
\end{proof}

\begin{prop}
\label{clique sizes}
We have:
\begin{enumerate}
\item For all $k$, $\omega\left( \cf_{\leqslant k} \right) \ge k+2$.
\label{case1}
\item For $k$ odd, $\omega \left( \cf_{\leqslant k} \right) \ge k+3$.
\label{case2}
\item For $k\equiv 8 \modu {12}$, $\omega \left( \cf_{\leqslant k} \right) \ge k+4$.
\label{case3}
\end{enumerate}
\end{prop}

\begin{proof}
These are immediate consequences of \autoref{clique construction1} and \autoref{clique construction2}:
For \eqref{case1}, we have $I(R_{k+1})=k$ and $|V_{R_{k+1}}|=k+2$; for \eqref{case2} taking $k=2n-1$ we have $I(S_n)=k$ and $|V_{S_n}|=k+3$; finally for \eqref{case3} if $k=12n-4$ then $I(T_n)=k$ and $|V_{T_n}| = k+4$.
\end{proof}

\autoref{up to k} follows immediately.

\begin{figure}
	\centering
	\begin{minipage}{.3\textwidth}
	\centering
	\includegraphics[height=7cm]{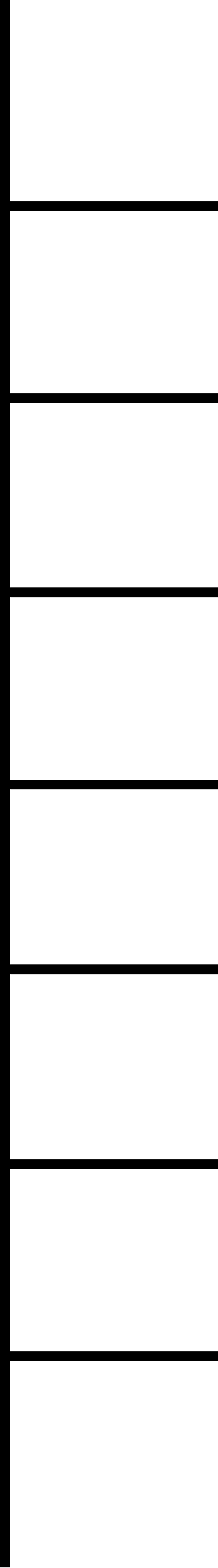}
	\subcaption{The subgraph $R_8$.}
	\label{Rn pic}
	\end{minipage}\hfill
	\begin{minipage}{.3\textwidth}
	\centering
	\includegraphics[height=7cm]{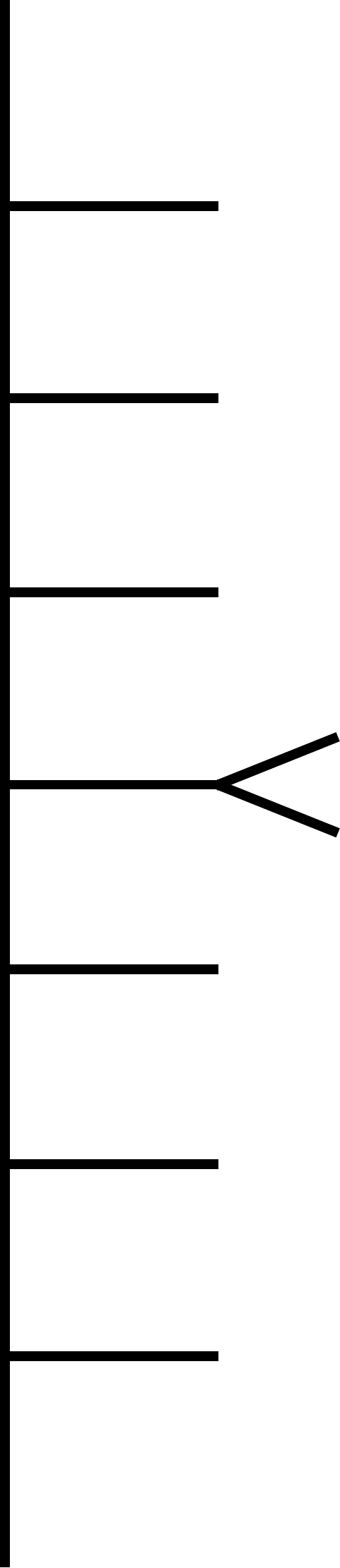}
	\subcaption{The subgraph $S_4$.}
	\label{Sn pic}
	\end{minipage}
	\begin{minipage}{.35\textwidth}
	\centering
	\includegraphics[height=7cm]{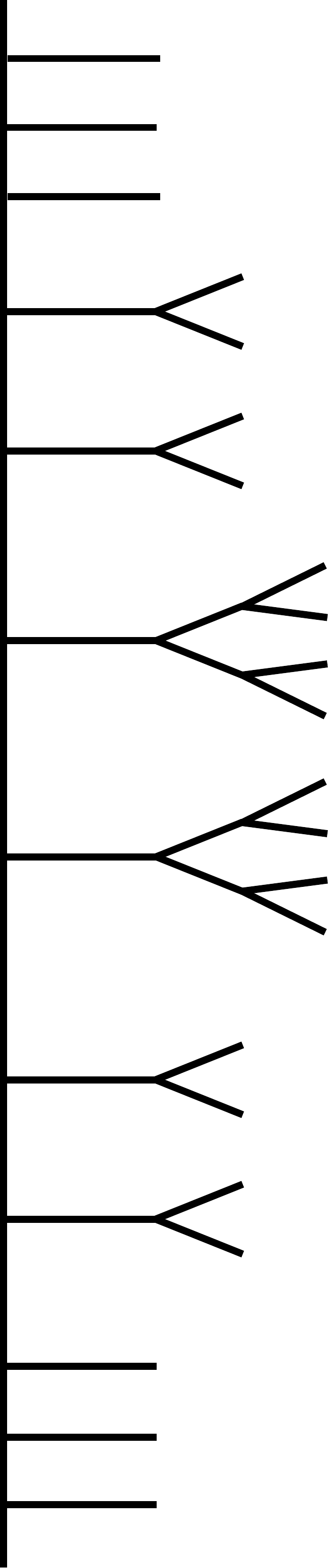}
	\subcaption{The subgraph $T_2$.}
	\label{Tn pic}
	\end{minipage}
	\caption{Subgraphs of $\cf$ with small maximum intersection.}
	\label{subgraph pics} 
\end{figure}

\subsection{Computational data for \texorpdfstring{$\omega(\cf_{\leqslant k})$}{w(F<=k)}}
There are many existing computational tools for computing graph invariants, including of course clique number. 
Using Mathematica to search through large finite subgraphs of $\cf_{\leqslant k}$, we obtained the data presented in \autoref{evidence}. 
The values of $k$ presented are exactly those, less than $100$, for which experimental data suggests that $\omega(\cf_{\leqslant k}) < 1+p(k)$. 
We make several comments about the data: \\
\begin{enumerate}
\item The list appears to include many primes that are two less than a composite.
Indeed, among $k\le 100$ the only exception is $k=37$. \\
\item While the values of the third column of \autoref{evidence} might be reduced in light of \autoref{improve Agol}, this observation will not resolve the gaps between the second and third columns. 
Indeed, for $k=7$, we have $\min\{ |\cl_r| : r>7\} =12 = 1+p(7)$, while the largest clique found experimentally in $\cf_{\leqslant  7}$ has size $10$. \\
\item It would be interesting to investigate the chromatic numbers $\chi(\cf_{\leqslant k})$ for the values of $k$ listed. However, initial attempts using Mathematica's software were insufficient to narrow the given bounds, even for $k=7$. \\
\item  For values of $k$ less than $100$ that are missing in the first column, experimental evidence produced cliques of size $1+p(k)$. For instance, we have $\omega(\cf_{\leqslant  24})=30$ (see \autoref{max clique 24}) and $\omega( \cf_{\leqslant  48})=54$. These represent the largest gaps observed (namely, size six) between $k$ and $\omega(\cf_{\leqslant k})$ (cf.~\autoref{big jumps}). \\
\end{enumerate}
We consider these phenomena to be worthy of further investigation. 

\vspace{.8cm}

\begin{figure}[H]
\begin{minipage}{\linewidth}
	\centering
	\includegraphics[width=6.5cm]{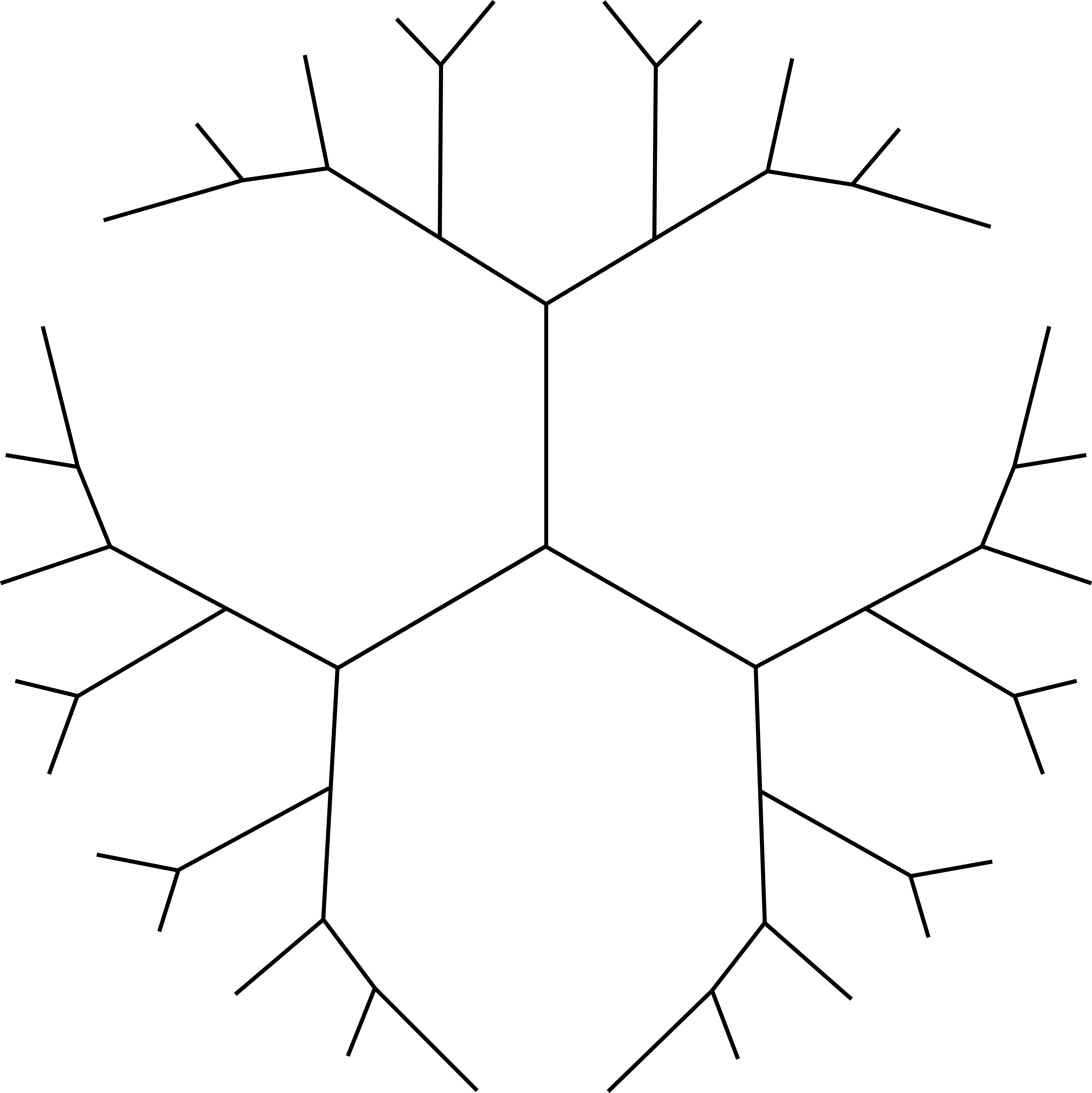}
	\caption{A clique of size $30=1+p(24)$ in $\cf_{\leqslant  24}$.}
	\label{max clique 24}

\vspace{1.2cm}

\begin{tabular}{|c|c|c|}
\hline
\multicolumn{1}{|c|}{$k$} & \multicolumn{1}{l|}{Largest clique} & \multicolumn{1}{l|}{Smallest coloring} \\ \hline
7                       & 10                                        & 12                                           \\ \hline
13                      & 16                                        & 18                                           \\ \hline
19                      & 23                                        & 24                                           \\ \hline
23                      & 27                                        & 30                                           \\ \hline
31                      & 34                                        & 38                                           \\ \hline
32                      & 36                                        & 38                                           \\ \hline
33                      & 37                                        & 38                                           \\ \hline
43                      & 46                                        & 48                                           \\ \hline
47                      & 51                                        & 54                                           \\ \hline
53                      & 57                                        & 60                                           \\ \hline
54                      & 59                                        & 60                                           \\ \hline
61                      & 65                                        & 68                                           \\ \hline
62                      & 67                                        & 68                                           \\ \hline
63                      & 67                                        & 68                                           \\ \hline
\end{tabular}
\quad
\hspace{.2cm}
\begin{tabular}{|c|c|c|}
\hline
\multicolumn{1}{|c|}{$k$} & \multicolumn{1}{l|}{Largest clique} & \multicolumn{1}{l|}{Smallest coloring} \\ \hline
67                      & 70                                        & 72                                           \\ \hline
73                      & 76                                        & 80                                           \\ \hline
74                      & 78                                        & 80                                           \\ \hline
75                      & 78                                        & 80                                           \\ \hline
79                      & 82                                        & 84                                           \\ \hline
83                      & 87                                        & 90                                           \\ \hline
84                      & 89                                        & 90                                           \\ \hline
85                      & 89                                        & 90                                           \\ \hline
89                      & 93                                        & 98                                           \\ \hline
90                      & 94                                        & 98                                           \\ \hline
91                      & 94                                        & 98                                           \\ \hline
92                      & 96                                        & 98                                           \\ \hline
93                      & 96                                        & 98                                           \\ \hline
97                      & 100                                       & 102                                           \\ \hline
\end{tabular}
\vspace{.2cm}
\caption{Exceptional values of $k$ that are less than 100 for which we are unable to decide whether $\omega(\cf_{\leqslant k}) = \chi(\cf_{\leqslant k})$.}
\label{evidence}
\end{minipage}
\end{figure}

\clearpage
\ \vspace{.5cm}\\
\newcommand{\etalchar}[1]{$^{#1}$}

\ \vspace{2cm}

\end{document}